\documentclass[12pt]{article}

\usepackage{amsthm,tikz,amssymb, latexsym, amsmath, amscd, amsfonts, array, graphicx, lmodern, slantsc,enumerate}
\usepackage[english]{babel}        
\usepackage[all]{xy}
\CompileMatrices

\usepackage{color}

\def\nl{\mathrm{P}^{\mathrm{op}}}
\def\P{\mathrm{P}}
\def\RP{\mathbb{R}\mathrm{P}}
\def\TC{\protect\operatorname{TC}}
\def\Sq{\protect\operatorname{Sq}}
\def\emb{\protect\operatorname{Emb}}
\def\ss{\protect\operatorname{SP}^2}
\def\imm{\protect\operatorname{Imm}}

\def\secat{\protect\operatorname{secat}}

\def\sb{\protect\operatorname{sb}}
\def\sbp{\overline{\protect\operatorname{sb}}}

\newtheorem{proposition}{Proposition}[section]
\newtheorem{corollary}[proposition]{Corollary}
\newtheorem{definition}[proposition]{Definition}
\newtheorem{theo}[proposition]{Theorem}
\newtheorem{remark}[proposition]{Remark}
\newtheorem{ejem}[proposition]{Example}
\newtheorem{lema}[proposition]{Lemma}

\begin{document}

\title{Symmetric bi-skew maps and \\ symmetrized motion planning in projective spaces}
\author{Jes\'us Gonz\'alez\thanks{Partially supported by Conacyt Research Grant 221221. }}
\date{\today}

\maketitle
\begin{abstract}
This work is motivated by the question of whether there are spaces $X$ for which the Farber-Grant symmetric topological complexity $\TC^S(X)$ differs from the Basabe-Gonz\'alez-Rudyak-Tamaki symmetric topological complexity $\TC^\Sigma(X)$. It is known that, for a projective space $\RP^m$, $\TC^S(\RP^m)$ captures, with a few potentially exceptional cases, the Euclidean embedding dimension of $\RP^m$. We now show that, for all $m\geq1$, $\TC^\Sigma(\RP^m)$ is characterized as the smallest positive integer $n$ for which there is a symmetric $\mathbb{Z}_2$-biequivariant map $S^m\times S^m\to S^n$ with a ``monoidal'' behavior on the diagonal. This result thus lies at the core of the efforts in the 1970's to characterize the embedding dimension of real projective spaces in terms of the existence of symmetric axial maps. Together with Nakaoka's description of the cohomology ring of symmetric squares, this allows us to compute both $\TC$ numbers in the case of $\RP^{2^e}$ for $e\geq1$. In particular, this leaves the torus $S^1\times S^1$ as the only closed surface whose symmetric (symmetrized) $\TC^S$ ($\TC^\Sigma$) -invariant is currently unknown.
\end{abstract}



\section{Introduction}
Farber's topological complexity of a space $X$, $\TC(X)$, can be defined as the sectional category\footnote{We follow the standard normalization convention for the sectional category: a fibration with a global section has zero sectional category.} of the double evaluation map $e_{0,1}\colon\P(X)\to X\times X$, i.e.~the fibration which sends a path $\gamma\colon[0,1]\to X$ into the ordered pair $e_{0,1}(\gamma)=(\gamma(0),\gamma(1))$. This concept, originally motivated by the motion planning problem in robotics (\cite{MR2074919}), has found interesting connections with classical problems in differential topology and homotopy theory. This paper develops on one such a connection. 

\medskip
A number of variants of Farber's $\TC$ concept have raised as models of the motion planning problem in the presence of symmetries. Such a line of research was opened up by Farber-Grant in~\cite{MR2359031} by considering the pullback (restriction) $\epsilon_{0,1}\colon \nl(X)\to X\times X-\Delta_X$ of $e_{0,1}$ under the inclusion $X\times X-\Delta_X\hookrightarrow X\times X$, where $\Delta_X=\{(x,x)\colon x\in X\}$ is the diagonal. Both $X\times X$ and $\P(X)$ come equipped with a natural switching involution, namely $\tau$: $\tau(x_1,x_2)=(x_2,x_1)$ and $(\tau\cdot\gamma)(t)=\gamma(1-t)$. The restricted involutions on $X\times X-\Delta_X$  and on the space of open paths $\nl(X)$ are fix-point free, and $\epsilon_{0,1}$ becomes a $\tau$-fibration.

\begin{definition}\label{jsuiwkdcn}
The symmetric topological complexity of a space $X$, $\TC^S(X)$, is one more than the $\tau$-equivariant sectional category of $\epsilon_{0,1}\colon\nl(X)\to X\times X-\Delta_X$.
\end{definition}

Thus, in the $\TC^S$-view, one considers motion planners (i.e.~local sections for $\epsilon_{0,1}$) for which the movement from an initial point $A$ to a final point $B$ (with $A\neq B$) is the time-reverse motion from $B$ to $A$. The part ``one more than'' in Definition~\ref{jsuiwkdcn} can be thought of as taking into account (a neighborhood of) the diagonal when describing actual symmetric motion planners on $X$.

\medskip
The fact that $\TC^S(X)$ is not a homotopy invariant of $X$ is one of the motivations for introducing in~\cite{bgrt} the following variant of Farber-Grant's $\TC^S$:

\begin{definition}
The symmetrized topological complexity of a space $X$, $\TC^\Sigma(X)$, is the smallest positive integer $n$ for which $X\times X$ can be covered by $n+1$ open sets $U$ each of which is closed under the switching involution $\tau$ on $X\times X$, and admits a continuous $\tau$-equivariant section $U\to\P(X)$ of the ($\tau$-equivariant) double evaluation map $e_{0,1}\colon\P(X)\to X\times X$.
\end{definition}

As noted in~\cite[Example~2.6]{markinprep}, $e_{0,1}$ is a $\tau$-fibration, so $\TC^\Sigma(X)$ can equivalently be defined as the $\tau$-equivariant sectional category of $e_{0,1}\colon\P(X)\to X\times X$.

\medskip
Much of the interest in $\TC^\Sigma(X)$ comes from the fact that, being a homotopy invariant of~$X$ (\cite[Proposition 4.7]{bgrt}), it differs from $\TC^S(X)$ by at most a unit. In fact, the inequalities
\begin{equation}\label{desigualdades}
\TC^S(X)-1\leq\TC^\Sigma(X)\leq\TC^S(X)
\end{equation}
hold for any reasonable space $X$ (see~\cite[Proposition~4.2]{bgrt}). 

\medskip
The equality $\TC^\Sigma(X)=\TC^S(X)$ is known to hold for a number of spaces: spheres (see \cite[Example 4.5]{bgrt} for even dimensional spheres, and~\cite{markinprep} for odd dimensional spheres), simply connected closed symplectic manifolds~(as follows from \cite[proof of Proposition~10]{MR2359031} and~\cite[Theorem~1]{MR1988783}; see~\cite[Theorem~6.1]{MR2491582} for the case of complex projective spaces), and all closed surfaces with the potential exceptional case of the torus $S^1\times S^1$ (see Remark~\ref{notadedon} below). But surprisingly, except for homotopically uninteresting situations~(\cite[Example~4.4]{bgrt}), \emph{no example of a space $X$ with $\TC^\Sigma(X)\neq\TC^S(X)$ is known.} This paper explores the differences between the two invariants in the case of a real projective space $\RP^m$ ---one of the central benchmarks in homotopy theory. Actually, the possibility of numerically telling apart $\TC^S(\RP^m)$ from $\TC^\Sigma(\RP^m)$ turns out to be amazingly subtle, complicated and closely related to a challenge with roots in Hopf's work~\cite{MR0004112}: A (hoped-for) characterization of the Euclidean embedding dimension of real projective spaces in terms of \emph{symmetric} axial maps. The task is best described by starting with a non-symmetric version of the problem.

\medskip
For a real projective space $\RP^m$ which is not parallelizable (i.e.~with $m\not\in\{1,3,7\}$), the invariant $\TC(\RP^m)$ is known to agree with the number $\imm(\RP^m)$ defined as the minimal dimension of Euclidean spaces where $\RP^m$ admits an immersion. In unrestricted terms, $\TC(\RP^m)$ agrees with the smallest positive integer $n$ for which there is an axial map $a\colon\RP^m\times\RP^m\to\RP^n$, i.e.~a map whose restriction to either of the axis is essential. Passing to universal covers, the above fact can be rephrased by saying the $\TC(\RP^m)$ is the smallest positive integer $n$ for which there is a map $b\colon S^m\times S^m\to S^n$ which is $\mathbb{Z}_2$-biequivariant,\footnote{The term ``bi-skew'' has been used in the literature as an alternative for ``$\mathbb{Z}_2$-biequivariant''.} i.e.~which satisfies $b(-x,y)=-b(x,y)=b(x,-y)$ for all $x,y\in S^m$ (see~\cite{MR0336757,MR1988783}).

\medskip
The $\TC$-$\imm$-axial phenomenon just described has a symmetric counterpart, summarized in~(\ref{equally}) and~(\ref{centralissue}) below. Let $\emb(\RP^m)$ stand for the smallest dimension of Euclidean spaces where $\RP^m$ admits an embedding. Let $\sb(m)$ stand for the smallest positive integer $n$ for which there exists a \emph{symmetric} axial map $\RP^m\times\RP^m\to\RP^n$, i.e.~an axial map which is $\mathbb{Z}_2$-equivariant with respect to the switching-axes involution $\tau$ on $\RP^m\times\RP^m$, and the trivial involution on $\RP^n$. Equivalently, $\sb(m)$ denotes the smallest positive integer $n$ for which there exists a $\mathbb{Z}_2$-biequivariant $b\colon S^m\times S^m\to S^n$ which is symmetric, i.e.~so that $b(x,y)=b(y,x)$ for all $x,y\in S^m$. 

\begin{remark}{\em
At the level of projective spaces, the difference $\sb(m)-m$ can be thought of as giving a measure of the failure of $\RP^m$ to be a (strictly) commutative $H$-space. We work with $\mathbb{Z}_2$-biequivariant maps, rather than with their axial-map counterpart, for the characterization of $\TC^\Sigma(\RP^m)$ (in Theorem~\ref{sigmavssb} below) is naturally given in terms of a slight specialization of the former maps.
}\end{remark} 

It is known that $\TC^S(\RP^m)\leq\emb(\RP^m)$ for any $m$, and in fact
\begin{equation}\label{equally}
\TC^S(\RP^m)=\emb(\RP^m),
\end{equation}
except \emph{possibly} for  $m\in\{6,7,11,12,14,15\}$ (see~\cite{MR3199728,MR2833574,MR2491582}). In addition, the main result in~\cite{MR0646069} asserts that $\emb(\RP^m)$ agrees, up to 1, with $\sb(m)$. Explicitly,
\begin{equation}\label{centralissue}\emb(\RP^m)-1\leq\sb(m)\leq\emb(\RP^m)\end{equation}
where the first inequality is asserted only if the ``metastable range'' condition $2\sb(m)>3m$ holds (e.g.~for $m>15$).

\medskip
To the best of our knowledge, the gap in~(\ref{centralissue}) has not been solved in either direction for general $m$. In fact, despite $\emb(\RP^m)$ has been studied extensively, no explicit projective space $\RP^m$ with $m>1$ and
\begin{equation}\label{destri}
\sb(m)<\emb(\RP^m)
\end{equation}
seems to have been singled out in the literature (but the slightly related Example~2 in~\cite[page~415]{MR0397750} should be noted). The problem can be approached via $\TC^\Sigma(\RP^m)$, which sits in a subtle way in between the two terms in~(\ref{destri}).
In fact, our main results (Theorems~\ref{sigmavssb} and~\ref{corop2} below) are motivated by comparing~(\ref{desigualdades}),~(\ref{equally}) and~(\ref{centralissue}), namely
\begin{eqnarray*}
\TC^S(\RP^m)-1\leq\TC^\Sigma(\RP^m)\leq\TC^S(\RP^m)\\
\TC^S(\RP^m)-1\leq{}\sb(m)\leq\TC^S(\RP^m)\hspace{3.3mm}
\end{eqnarray*}
(the second chain of inequalities holding, say, for $m>15$).

\begin{theo}\label{sigmavssb} For $m\geq1$,
$\sb(m)\leq\sbp(m)=\TC^\Sigma(\RP^m)\leq\TC^S(\RP^m)\leq\emb(\RP^m)$.
\end{theo}

The function $\sbp$ will be defined later in the paper (Definition~\ref{ladesbp} below). For now, it suffices to remark that the quality that distinguishes $\sbp$ from $\sb$ is that, in the definition of $\sbp$, symmetric $\mathbb{Z}_2$-biequivariant maps are required to have a reasonably well-controlled behavior on the diagonal. 

\medskip
Unlike~(\ref{equally}) and~(\ref{centralissue}), the characterization of $\TC^\Sigma(\RP^m)$ in Theorem~\ref{sigmavssb} holds without restrictions on~$m$. Loosely speaking, Theorem~\ref{sigmavssb} asserts that the $\TC^\Sigma$-analogue of~(\ref{destri}) can be ruled out effectively by strengthening slightly the concept of symmetric $\mathbb{Z}_2$-biequivariant maps. Additionally, it should be stressed that, for most values\footnote{The inequality $\sb(m)\geq\emb(\RP^m)-1$ is currently known to hold for all values of m, except possibly $m\in\{5,6,7,9,11,12,15\}$. Further, if attention is restricted to the $\TC^S$-$\sb$ relationship, then it is worth noticing that the inequality $\sb(m)\geq\TC^S(\RP^m)-2$ is currently known to hold for all values of $m$, except possibly $m=12$.} of $m$, at most one of the three inequalities in the conclusion of Theorem~\ref{sigmavssb} fails to be an equality ---the subtle point being the possibility that the potential failing inequality would depend on $m$. 

\begin{theo}\label{corop2}
All three inequalities in the conclusion of Theorem~\ref{sigmavssb} are sharp provided $m=2^e$ with $e\geq1\!:$
$
\emb(\RP^{2^e})=\TC^S(\RP^{2^e})=\TC^\Sigma(\RP^{2^e})=\sbp(2^e)=\sb(2^e)=2^{e+1}.$
\end{theo}

Theorem~\ref{corop2} should be compared with the fact that $\TC(\RP^{2^e})=\imm(\RP^{2^e})=2^{e+1}-1$, for $e\geq1$.

\begin{remark}\label{chiquito}{\em The case $e=0$ in Theorem~\ref{corop2} is indeed exceptional in that, while $\sb(1)=1$  is obvious (multiplication of complex numbers of norm 1), the equality $\TC^\Sigma(S^1)=2$ is asserted in~\cite{doninprep,markinprep} after subtle considerations. In the final section of this paper we offer a streamlined proof of the equality $\TC^\Sigma(S^1)=2$.
}\end{remark}

It is tempting to think of the agreement between $\TC^\Sigma$ and its monoidal version (asserted in~\cite[Theorem 5.2]{markinprep}) as indirect evidence for the possibility that $\overline{\sb}(m)=\sb(m)$. However, this would have to be taken with care in view of Remark~\ref{chiquito}. 

\medskip
We do not expect the equality $\TC^S=\TC^\Sigma$ in Theorem~\ref{corop2} to be generic; we believe that the equality $\TC^S(X)=\TC^\Sigma(X)$ would have to fail even for reasonably well-behaved spaces $X$. In other words, it is hard to think that considering a neighborhood of the diagonal on its own would have to lead to the most efficient way to symmetrically motion plan. It would be interesting if the equality $\TC^S=\TC^\Sigma$ actually failed for some $\RP^m$, for then the inequality $\emb(\RP^m)\neq\sb(m)$ would be forced.

\medskip
The author would like to thank Mark Grant and Kee Lam for illuminating email discussions on the topics of this paper, and Don Davis and Mark Grant for sharing with the author of this paper early versions of their preprints~\cite{doninprep,markinprep}.

\section{$\TC^S$, $\TC^\Sigma$ and equivariant partitions of unity}
Although the inequality $\sb(m)\leq\TC^\Sigma(\RP^m)$ in Theorem~\ref{sigmavssb} follows easily from the asserted characterization $\sbp(m)=\TC^\Sigma(\RP^m)$, it is convenient to start with:
\begin{proof}[Proof of the inequality $\sb(m)\leq\TC^\Sigma(\RP^m)$ in Theorem~\ref{sigmavssb}]
Let $U_0,\ldots,U_n$ be a covering of $\RP^m\times\RP^m$ (say $n=\TC^\Sigma(\RP^m)$) by open sets each of which:
\begin{itemize}
\item is closed under the swapping involution $\tau((L_1,L_2))=(L_2,L_1)$ of $\RP^m\times\RP^m$, and
\item admits a $\tau$-equivariant section $s_i\colon U_i\to\P(\RP^m)$ for the double evaluation map $e_{0,1}\colon\P(\RP^m)\to\RP^m\times\RP^m$. (Recall from the introduction that $\tau$ acts on the path space $\P(\RP^m)$ by $\tau(\gamma)(t)=\gamma(1-t)$.)
\end{itemize}

Take a $\tau$-equivariant partition of unity $\{h_i\}$ subordinate to the cover $\{U_i\}_i$, i.e.~a family of continuous functions $h_i\colon \RP^m\times \RP^m\to [0,1]$, $0\leq i\leq n$, satisfying:
\begin{enumerate}[(i)]
\item $h_i(L_1,L_2)=h_i(L_2,L_1)$, for $L_1,L_2\in\RP^m$;\label{partcond1}
\item $\mbox{supp}(h_i)\subseteq U_i$;
\item $\max\{h_i(L_1,L_2)\colon 0\leq i\leq n\}=1$, for each $(L_1,L_2)\in\RP^m\times\RP^m$.\label{partcond3}
\end{enumerate}
For the existence of such a partition see, for instance,~\cite[Lemma~3.2]{markinprep}, \cite[page~321]{MR0126506} or, more generally, \cite[Theorem~5.2.5]{MR972503}.

\medskip
Recall the factorization
$\P(\RP^m)\stackrel{f}{\longrightarrow}S^m\times_{\mathbb{Z}_2}S^m\stackrel{\pi}\longrightarrow\RP^m\times\RP^m$ of $e_{0,1}$, where
\begin{itemize}
\item the middle space is the Borel construction $S^m\times S^m\left/\rule{0mm}{3.5mm}(x,y)\sim(-x,-y)\right.$;
\item $f(\gamma)$ is the class of the pair $(\widetilde{\gamma}(0),\widetilde{\gamma}(1))$, where $\widetilde{\gamma}$ is a lifting of $\gamma\colon[0,1]\to\RP^m$ through the usual double covering $S^m\to\RP^m$;
\item $\pi([(x_1,x_2)])=(L_{x_1},L_{x_2})$, where $L_{x_i}$ is the line determined by $x_i$.
\end{itemize}
The maps $\sigma_i:=f\circ s_i$ are $\tau$-equivariant local sections of $\pi$, where $\tau$ acts on $S^m\times_{\mathbb{Z}_2}S^m$ by $\tau\cdot([(x,y)])=[(y,x)]$. Since $\pi$ is a $\mathbb{Z}_2$-principal fibration, where the generator $g$ of $\mathbb{Z}_2$ acts on $S^m\times_{\mathbb{Z}_2}S^m$ via the formula
\begin{equation}\label{gaccion}
g\cdot[(x,y)]=[(-x,y)]=[(x,-y)],
\end{equation}
$\sigma_i$ yields a trivialization of the restriction of $\pi$ to $U_i$, i.e.~a $\mathbb{Z}_2$-equivariant homeomorphism $\lambda_i\colon \pi^{-1}(U_i)\to\mathbb{Z}_2\times U_i$ characterized by the condition
\begin{equation}
\lambda_i(x)=(g^\epsilon,\pi(x)),\mbox{ \ where $\epsilon\in\{0,1\}$ and $x=g^{\epsilon}\cdot\sigma_i(\pi(x))$.}
\end{equation}
Note that $\mathbb{Z}_2\times U_i$ inherits a $\tau$-involution via $\lambda_i$; in fact, since the action~(\ref{gaccion}) commutes with that of $\tau$, we see that this inherited $\tau$-involution on $\mathbb{Z}_2\times U_i$ takes the form
\begin{equation}\label{compatibles}
\tau\cdot(g^\epsilon,(L_1,L_2))=(g^\epsilon,\tau\cdot(L_1,L_2))=(g^\epsilon,(L_2,L_1)).
\end{equation}

Let $C\mathbb{Z}_2$ stand for the cone $\mathbb{Z}_2\times [0,1]/(g,0)\sim(g^2,0)$ ---an interval $[0,1]$ in disguise. As observed in~\cite[page~87]{Schwarz66}, the composition of $\lambda_i$ with the map $\mu_i\colon \mathbb{Z}_2\times U_i\to C\mathbb{Z}_2$ given by $\mu_i(g^\epsilon,u)=[(g^\epsilon,h_i(u))]$ extends to a (continuous) map $\Lambda_i\colon S^m\times_{\mathbb{Z}_2}S^m\to C\mathbb{Z}_2$ which is $\mathbb{Z}_2$-equivariant ($g$ acts ``horizontally'' on the cone: $g\cdot[(g^\epsilon,t)]=[(g^{1+\epsilon},t)]$). Further, in the present situation,
\begin{enumerate}[(i)]\addtocounter{enumi}{3}
\item $\Lambda_i$ is $\tau$-invariant (i.e.~$\Lambda_i([x,y])=\Lambda_i([y,x])$ for $x,y\in S^m$), in view of~(\ref{partcond1}) and~(\ref{compatibles}).\label{invariant}
\end{enumerate}
Of course Schwarz's goal is to obtain that, in view of~(\ref{partcond3}), the product $\prod_i\Lambda_i$ yields a $\mathbb{Z}_2$-equivariant map $\Lambda\colon S^m\times_{\mathbb{Z}_2}S^m\to\left(\mathbb{Z}_2\right)^{\ast (n+1)}=S^n$, which is $\tau$-invariant in view of~(\ref{invariant}). Consequently, the composition of the canonical projection $S^m\times S^m\to S^m\times_{\mathbb{Z}_2}S^m$ with $\Lambda$ yields a symmetric $\mathbb{Z}_2$-biequivariant map, completing the proof of Theorem~\ref{sigmavssb}.
\end{proof}

The proof above can be used to show the strengthened inequality $\sbp(m)\leq\TC^\Sigma(\RP^m)$ in Theorem~\ref{sigmavssb}. We first give the precise definition of the number $\sbp(m)$.

\begin{definition}\label{ladesbp}
$\sbp(m)$ is the smallest positive integer~$n$ for which there is a symmetric $\mathbb{Z}_2$-biequivariant map $b\colon S^m\times S^m\to S^n$ with the property that the image under $b$ of the diagonal $\Delta_{S^m}=\{(x,x)\colon x\in S^m \}\subset S^m\times S^m$ does not intersect some $n$-dimensional subspace of $\mathbb{R}^{n+1}$.
\end{definition}

Of course, after a suitable rotation of $S^n$, we can assume that the first Euclidean coordinate $b_0\colon S^m\times S^m\to \mathbb{R}$ of the map $b=(b_0,b_1,\ldots,b_n)$ in Definition~\ref{ladesbp} satisfies $b_0(x,x)>0$ for all $x\in S^m$.

\begin{proof}[Proof of the inequality $\sbp(m)\leq\TC^\Sigma(\RP^m)$ in Theorem~\ref{sigmavssb}]
In view of~\cite[Theorem~5.2]{markinprep}, we can start with an open covering $U_0,\ldots,U_n$ of $\RP^m\times\RP^m$ (say $n=\TC^\Sigma(\RP^m)$) so that each $U_i$:
\begin{itemize}
\item contains the diagonal $\Delta_{\RP^m}$, 
\item is closed under the swapping involution $\tau((L_1,L_2))=(L_2,L_1)$ of $\RP^m\times\RP^m$, and
\item admits a $\tau$-equivariant section $s_i\colon U_i\to\P(\RP^m)$ for the double evaluation map $e_{0,1}\colon\P(\RP^m)\to\RP^m\times\RP^m$ such that, for all $L\in\RP^m$, $s_i(L,L)$ is the constant path at $L$.
\end{itemize}
We then proceeding as in the previous proof, to find that $\sigma_i(L,L)=[(x,x)]$ whenever $x\in L$, so that $\lambda_i([(x,x)])=(g^0,(L,L))$ for all $x\in S^m$. This immeditely implies that the resulting symmetric $\mathbb{Z}_2$-biequivariant map
$$
S^m\times S^m \to S^m\times_{\mathbb{Z}_2} S^m \stackrel{\Lambda}\longrightarrow S^n=(\mathbb{Z}_2)^{*(n+1)}\subset\prod_{i=0}^{n}C\mathbb{Z}_2
$$
sends the diagonal $\Delta_{S^m}$ into the simplex generated by the various neutral elements $g^0$ of each factor $C\mathbb{Z}_2$.
\end{proof}

\begin{proof}[Proof of the inequality $\sbp(m)\leq\TC^S(\RP^m)$ in Theorem~\ref{sigmavssb}]
Let $n=\TC^s(\RP^m)$ and pick a covering $U_1,\ldots U_n$ of $\RP^m\times\RP^m-\Delta_{\RP^m}$ by open sets which are closed under the switching-axes involution $\tau$, each with a $\tau$-equivariant section $s_i\colon U_i\to\P(\RP^m)$ for the double evaluation map $e_{0,1}$. As noted in the proof of~\cite[Corollary~9]{MR2359031}, we can also pick an open neighborhood $U_0$ of $\Delta_{\RP^m}$ in $\RP^m\times\RP^m$, which is closed under the action of $\tau$, together with a $\tau$-equivariant section $s_0\colon U_0\to\P(\RP^m)$ of $e_{0,1}$ with the property that, for each line $L\in\P(\RP^m)$, $s_0(L,L)$ is the constant path at $L$. Then we are in the situation at the start of the proof of the inequality $\sb(m)\leq\TC^\Sigma(\RP^m)$, so we apply the same constructions (using the same notation), except that this time we can assume the additional property that the $\tau$-equivariant partition of unity satisfies $h_0(L,L)=1$ for all $L\in\RP^m$. In such a setting, it follows that $\Lambda_0([x,x])=(g^0,1)$ and $\Lambda_i([x,x])=(g^0,0)=(g,0)$ for all $x\in S^m$ and all $i>0$. Therefore the resulting $\Lambda\colon S^m\times_{\mathbb{Z}_2}S^m\to S^n$ is now constant on points of the form $[x,x]$, and the corresponding symmetric $\mathbb{Z}_2$-biequivariant map $S^m\times S^m\to S^n$ is constant on the diagonal $\Delta_{S^m}$.
\end{proof}

The proof of Theorem~\ref{sigmavssb} will be complete once we show (in the next section) the inequality $\TC^\Sigma(\RP^m)\leq\sbp(m)$. (In view of~(\ref{desigualdades}), the proof we have just given for the inequality $\sbp(m)\leq\TC^S(\RP^m)$ can be waived; we included the additional idea in support of Remark~\ref{lasime} below.)

\section{Symmetrized motion rules}
Definition~\ref{ladesbp} allows us to apply, word for word, the proof of~\cite[Proposition~6.3]{MR1988783} in order to complete the proof of Theorem~\ref{sigmavssb}. This short section includes the easy details for completeness.

\begin{proof}[Proof of the inequality $\TC^\Sigma(\RP^m)\leq\sbp(m)$ in Theorem~\ref{sigmavssb}] Let $n=\sbp(m)$ and pick a symmetric $\mathbb{Z}_2$-biequivariant map $b=(b_0,\ldots,b_n)\colon S^m\times S^m\to S^n$ such that 
\begin{equation}\label{condFTY}
\mbox{$b_0(x,x)>0$ for all $x\in S^m$.}
\end{equation}
For $0\leq i\leq n$, set $V'_i=V_i-\Delta_{\RP^m}$ where $V_i$ is the image under the projection $\pi\colon S^m\times S^m\to\RP^m\times \RP^m$ of the set $U_i=\{(x,y)\in S^m\times S^m\colon b_i(x,y)\neq0\}$. All sets $U_i$, $V_i$, and $V'_i$ are open, and closed under the action of the corresponding switching-axes involutions $\tau$. Furthermore, $\tau$-equivariant (continuous) sections $s_i\colon V'_i\to\P(\RP^m)$ for the double evaluation map $e_{0,1}\colon\P(\RP^m)\to\RP^m\times\RP^m$ are defined as follows: For $(L_1,L_2)\in V'_i$ there are four pairs $(\pm x_1,\pm x_2)\in U_i\cap\pi^{-1}(L_1,L_2)$. Only two of these, say $(x_1,x_2)$ and $(-x_1,-x_2)$, have positive image under $b_i$. We then set $s_i(L_1,L_2)$ to be the path in $\P(\RP^m)$ corresponding to the rotation from $L_1$ to $L_2$, through the plane these lines generate, so that $x_1$ rotates toward $x_2$ through an angle less than $180^\circ$. As illustrated below, the resulting path $s_i(L_1,L_2)$ does not depend on whether $(x_1,x_2)$ or $(-x_1,-x_2)$ is used.
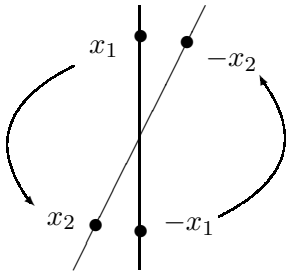
\begin{center}
\begin{picture}(0,107)(0,-51)
\put(0,0){\line(0,1){50}}
\put(0,0){\line(0,-1){50}}
\put(0,0){\line(1,2){25}}
\put(0,0){\line(-1,-2){25}}
\put(-2.5,35.5){$\bullet$}
\put(15,33){$\bullet$}
\put(-2.5,-38.5){$\bullet$}
\put(-19.5,-36){$\bullet$}
\put(-19,32){\small $x_1$}
\put(9,-35){\small $-x_1$}
\put(-35,-33){\small $x_2$}
\put(25,27){\small $-x_2$}
\qbezier(-25,27)(-65,7)(-41,-25)
\put(-39.5,-26){\vector(1,-1){0}}
\qbezier(30,-30)(70,-10)(46,22)
\put(45.4,23.8){\vector(-1,2){0}}
\end{picture}
\end{center}
Because of~(\ref{condFTY}), $s_0$ extends to a continuous $\tau$-equivariant section of $e_{0,1}$ on $V_0$ so that $s_0(L,L)$ is the constant path (with constant value $L$). The proof is complete since $V_0$, $V_1$, \ldots, $V_n$ cover $\RP^m\times\RP^m$.
\end{proof}

In view of~(\ref{centralissue}), Theorem~\ref{sigmavssb} implies that instances with
\begin{equation}\label{diferentes}
\TC^\Sigma(\RP^m)<\TC^S(\RP^m)
\end{equation}
could only happen when optimal embeddings of $\RP^m$ are not realizable by symmetric axial maps ---a possibility that, to the best of our knowledge, cannot be currently overruled for $m>1$. Furthermore, the equalities $\TC^\Sigma(\RP^m)=\sbp(m)=\sb(m)$ would be forced whenever~(\ref{diferentes}) holds (here we are implicitly assuming that $m$ lies in the range where the first inequality in~(\ref{centralissue}) holds).

\begin{remark}\label{lasime}{\em
A close look at the techniques in this and the previous section reveals that, for any $m\geq1$, $\TC^\Sigma(\RP^m)$ agrees with the smallest positive integer $n$ for which there is a symmetric $\mathbb{Z}_2$-biequivariant map $S^m\times S^m\to S^n$ which is constant on the diagonal. The later fact is the right symmetrization of the corresponding property for $\TC(\RP^m)$, though the proof in the non-symmetric case reduces to the simpler homotopy fact that an axial map $\RP^m\times\RP^m\to\RP^n$, being nulhomotopic on the diagonal, is homotopic to a (necessarily axial) map $\RP^m\times\RP^m\to\RP^n$ which is in fact constant on the diagonal.
}\end{remark}

\section{Symmetric squares and $\TC^\Sigma$}
The general inequalities $\emb(\RP^m)\geq\TC^\Sigma(\RP^m)\geq\sb(m)$ can be used to compute the value of $\TC^\Sigma(\RP^m)$ provided one can settle suitably large lower bounds for $\sb(m)$. In this section we start by establishing such estimates in the case $m=2^e\geq2$, thus proving Theorem~\ref{corop2}. The method (based on a Borsuk-Ulam-type argument using symmetric squares) is first illustrated in Example~\ref{sp2p2} below for the (geometrically much simpler) case $e=1$. 

\medskip
The symmetric square of a space $X$, $\ss(X)$, is the orbit space of $X\times X$ by the switching involution $\tau$. We think of $X$ as being embedded (diagonally) both in $X\times X$ and in $\ss(X)$. Note also that any symmetric axial map $\RP^m\times\RP^m\to\RP^n$ factors through the canonical projection $\RP^m\times\RP^m\to\ss(\RP^m)$. A useful geometric fact is that $\mathrm{SP}^2(\RP^2)$ is homeomorphic to $\RP^4$ with the diagonal inclusion $\RP^2\hookrightarrow\ss(\RP^2)=\RP^4$ being nullhomotopic (see~\cite[Lemma~1]{MR0341511}). 

\begin{ejem}\label{sp2p2}{\em
In view of Theorem~\ref{sigmavssb} and the well known equality $\emb(\RP^2)=4$, the case $e=1$ in Theorem~\ref{corop2} will follow once we show $\sb(2)\geq4$. So, assume for a contradiction that the composition $$b=\left(\RP^2\times\RP^2\to\ss(\RP^2)=\RP^4\stackrel{b_1}{\longrightarrow}\RP^3\right)$$ is a symmetric axial map. The ``axial'' condition gives $ b^*(x)=x\otimes1+1\otimes x$, where $x\in H^1(\RP^m;\mathbb{Z}_2)$ stands for the generator. This forces $b_1^*(x)=x$, which is impossible as $x^4=0$ on $\RP^3$, but $x^4\neq0$ in $\RP^4$.
}\end{ejem}

\begin{remark}\label{notadedon}{\em
Davis' observation~(\cite{davistcsyms1}) that the assertions
\begin{enumerate}[(i)]
\item $\TC(X)\leq\TC^\Sigma(X)\leq\TC^S(X)$;
\item all closed surfaces $\Gamma$ have $\TC^S(\Gamma)\leq4$~(\cite[Proposition~10]{MR2359031});
\item except for $S^2$, $S^1\times S^1$ and $\RP^2$, all closed surfaces $\Gamma$ have $\TC(\Gamma)=4$~(\cite{CV,dran,MR3544546, Far}),
\end{enumerate}
imply that both inequalities in~(i) above are in fact equalities for all closed surfaces $\Gamma$, except perhaps for $\Gamma\in\{S^2, S^1\times S^1, \RP^2\}$. The corresponding equality $\TC^\Sigma(\RP^2)=\TC^S(\RP^2)$ is now accounted for by Example~\ref{sp2p2}. As noted in the introduction of this paper, the equality $\TC^S(S^2)=\TC^\Sigma(S^2)$ is also known. (See Example~\ref{tcstoro} below for a discussion of what is currently known in the case of the torus.)
}\end{remark}

The topology of symmetric squares $\mathrm{SP}^2(\RP^m)$ for $m>2$ is much more subtle than that for $m=2$. In order to deal with the general form of Theorem~\ref{corop2}, we shall make use of the description in~\cite{MR0091462} of the mod 2 cohomology ring of $\mathrm{SP}^2(X)$. We give a short description of Nakaoka's results after stating the main goal in this section, Proposition~\ref{altura} below, and observing that it yields Theorem~\ref{corop2}. The proof of Proposition~\ref{altura} will then follow.

\begin{proposition}\label{altura}
Let $m\geq2$ with $2^e\leq m<2^{e+1}$. Then $H^1(\mathrm{SP}^2(\RP^m);\mathbb{Z}_2)=\mathbb{Z}_2$, and the generator $\phi_1$ of this group satisfies $\phi_1^{2^{e+1}}\neq0=\phi_1^{2^{e+1}+1}$.
\end{proposition}

Since $\emb(\RP^{2^e})=2^{e+1}$ is well known, it is clear that Proposition~\ref{altura} is all that is needed to have the argument in Example~\ref{sp2p2} prove the general case of Theorem~\ref{corop2}.

\medskip
Here is a brief summary of Nakaoka's description of the mod 2 cohomology ring  of $\ss(X)$ for a finite 0-connected polyhedron~$X$~(\cite{MR0091462}). Through the rest of the paper, cochain complexes and cohomology are taken with coefficients mod 2. 

\medskip
The identity and the involution $\tau$ induce maps at the cochain level $C^*(X\times X,X)$, and we let $\sigma\colon C^*(X\times X,X)\to C^*(X\times X,X)$ stand for the corresponding difference morphism. Note that the kernel and the image of $\sigma$ agree; we let ${}^\sigma C^*(X\times X,X)$ stand for the resulting cochain subcomplex, writing ${}^\sigma H^*(X\times X,X)$ for its cohomology. The so-called Smith-Richardson short exact sequence $$0\to {}^\sigma C^*(X\times X,X) \to C^*(X\times X,X)\to {}^\sigma C^*(X\times X,X)\to0$$ yields a connecting morphism $\partial\colon {}^\sigma H^*(X\times X,X)\to {}^\sigma H^{*+1}(X\times X,X)$. Since the canonical projection $(X\times X,X)\to(\ss(X),X)$ identifies the cochain complexes $C^*(\ss(X),X)$ and ${}^\sigma C^*(X\times X,X)$, we get a morphism $\nu\colon H^*(\ss(X),X)\to H^{*+1}(\ss(X),X)$ corresponding to $\partial$. Then, morphisms $E_s\colon H^*(X)\to H^{*+s}(\ss(X),X)$ are defined for $s\geq1$ as the composition
$$
E_s=\left(H^*(X)\stackrel{\delta}\longrightarrow H^{*+1}(\ss(X),X)\stackrel{\nu^{s-1}}\longrightarrow H^{*+s}(\ss(X),X)\right),
$$
where $\delta$ is the usual connecting map associated to the pair $(\ss(X),X)$. On the other hand, note that the transfer map $C^*(X\times X)\to C^*(\ss(X))$ lands in the relative cochain subcomplex $C^*(\ss(X),X)$ thus defining a morphism $\phi\colon H^*(X\times X)\to H^*(\ss(X),X)$. Lastly, by restricting under the inclusion of pairs $(X,\varnothing)\hookrightarrow(\ss(X),X)$, we get corresponding maps $H^*(X)\to H^{*+s}(\ss(X))$ and $H^*(X\times X)\to H^*(\ss(X))$, which will also be denoted by $E_s$ and $\phi$, respectively (the context will clarify which map we refer to).

\medskip
The results we need from Nakaoka's work~\cite{MR0091462} are packed in the following omnibus result:

\begin{theo}\label{nakaoka} Fix a homogeneous basis $\{b_0, b_1,\ldots,b_m\}$ of $H^*(X)$. Let $R$ stand for either $\varnothing$ or $X$, and set
$$
\ell=\begin{cases}1,& \mbox{if $R=X;$}\\2, & \mbox{if $R=\varnothing$.}\end{cases}
$$
A basis for $H^*(\ss(X),R)$ consists of 1, the elements $E_s(b_i)$ with $\ell\leq s\leq \deg(b_i)$, and the elements $\phi(b_i\otimes b_j)$ with $i<j$. The ring structure is determined by the two relations:
\begin{enumerate}[(a)]
\item\label{productos} $\phi(b_i\otimes b_j)\cdot\phi(b_u\otimes b_v)=\phi((b_i\cdot b_u)\otimes(b_j\cdot b_v))+\phi((b_i\cdot b_v)\otimes(b_j\cdot b_u))$.
\item $E_s(b_i)\cdot\phi(b_u\otimes b_v)=E_s(b_i)\cdot E_t(b_j)=0$.
\end{enumerate}
The right-hand side in~(\ref{productos}) can be expanded in terms of basis elements by repeated applications of the relations:
\begin{enumerate}[(a)]\addtocounter{enumi}{2}
\item $\phi(b_j\otimes b_i)=\phi(b_i\otimes b_j)$.
\item $\phi(b_i\otimes b_i)=\sum_{s=\ell}^{\deg(b_i)}E_s(\Sq^{\deg(b_i)-s}b_i)$. 
\item $E_{\deg(b_i)+k}(b_i)=\sum_{s=\max(k,\ell)}^{\deg(b_i)+k-1}E_s(\Sq^{\deg(b_i)+k-s}b_i)$, for $k\geq1$.
\end{enumerate}
The action of the Steenrod algebra is determined by the relations:
\begin{enumerate}[(a)]\addtocounter{enumi}{5}
\item $\Sq^k\phi(b_i\otimes b_j)=\phi\Sq^k(b_i\otimes b_j)+\sum_{s=\ell}^kE_s(\Sq^{k-s}(b_i\cdot b_j))$.
\item $\Sq^kE_s(b_i)=\sum_{j=0}^{k}\binom{s-1}{k-j}E_{k+s-j}(\Sq^jb_i)$, for $\ell\leq s\leq\deg(b_i)$.
\end{enumerate}
\end{theo}

Of course, Theorem~\ref{nakaoka} is most useful when we actually know the structure of $H^*(X)$ as an algebra over the mod 2 Steenrod algebra, and we then get a full description of $H^*(\ss(X),R)$ as an algebra over the mod 2 Steenrod algebra.

\begin{ejem}\label{nakatoro}{\em
A basis for the mod 2 cohomology of the torus $T=S^1\times S^1$ consists of the elements $1$, $x$, $y$ and $xy$ (with trivial action of the Steenrod algebra), where $x$ and $y$ are 1-dimensional clases. Then a basis for the mod 2 cohomology of $\ss(T)$ is given by 1, $\phi(1\otimes x)$, $\phi(1\otimes y)$, $\phi(1\otimes xy)$, $\phi(x\otimes y)$, $\phi(x\otimes xy)$, $\phi(y\otimes xy)$ and $E_2(xy)$. Further, by straightforward calculation we check that the only non-vanishing products are
\begin{align*}
&\phi(1\otimes x)\phi(1\otimes y)=\phi(1\otimes xy)+\phi(x\otimes y);\\
&\phi(1\otimes xy)^2=\phi(x\otimes y)^2=\phi(xy\otimes xy)=E_2(xy);\\
&\phi(1\otimes x)\phi(1\otimes y)\phi(x\otimes y)=\phi(xy\otimes xy)=E_2(xy);\\
&\phi(1\otimes x)\phi(1\otimes y)\phi(1\otimes xy)=\phi(xy\otimes xy)=E_2(xy).
\end{align*}
}\end{ejem}

\begin{proof}[Proof of Proposition~\ref{altura}]
We use Theorem~\ref{nakaoka} with the obvious basis $\{1,x,x^2,\ldots,x^m\}$ of $H^*(\RP^m)$ and $R=\varnothing$ (so $\ell=2$). Recall $\Sq^sx^i=\binom{i}{s}x^{i+s}$. Then the only basis element of degree 1 in $H^*(\ss(X))$ is $\phi_1=\phi(1\otimes x)$, which by direct calculation has
\begin{align*}
\phi_1^2=&\phi(1\otimes x^2)+\phi(x\otimes x)=\phi(1\otimes x^2), \\
\phi_1^4=&\phi(1\otimes x^4)+\phi(x^2\otimes x^2)=\phi(1\otimes x^4)+E_2(x^2).
\end{align*}
Assuming $\phi_1^{2^i}=\phi(1\otimes x^{2^i})+E_{2^{i-1}}(x^{2^{i-1}})$, we get
$$
\phi_1^{2^{i+1}}=(\phi(1\otimes x^{2^i})+E_{2^{i-1}}(x^{2^{i-1}}))^2=\phi(1\otimes x^{2^{i+1}})+\phi(x^{2^i}\otimes x^{2^i}),
$$
which by standard properties of mod 2 binomial coefficients (and, of course, Theorem~\ref{nakaoka}) implies $$\phi_1^{2^{i+1}}=\phi(1\otimes x^{2^{e+1}})+E_{2^i}(x^{2^i}).$$ The conclusion of Proposition~\ref{altura} follows from the  $i=e$ case of the last equality.
\end{proof}

The arguments in this section suggest that a systematic analysis of the (rich but not yet fully explored) algebraic topology properties of the symmetric square $\mathrm{SP}^2(\RP^m)$ could have implications on (and lead to a better understanding of) $\TC^\Sigma(\RP^m)$ and $\sb(m)$. For instance, Mark Grant has noticed that $\TC^\Sigma(X)$ is bounded from below by the cup-length of $H^*(\ss(X))$ ---compare to Proposition~\ref{markscuplength} below. The latter observation can be used (with $X=S^1$, see Corollary~\ref{circle} below) to reprove, in a slightly streamlined way, the fact (first noticed in~\cite{doninprep,markinprep}) that $\TC^\Sigma(S^1)=2$. 

\medskip
The following result is basically~\cite[Theorem~4.5]{markinprep}. We offer a slightly more conceptual proof.
\begin{proposition}\label{markscuplength}
$\TC^\Sigma(X)$ is bounded from below by the sectional category of the diagonal inclusion $X\hookrightarrow\ss(X)$.
\end{proposition}
\begin{proof}
Consider the diagram
$$\xymatrix{
X \ar@{^{(}->}[r] \ar@{_{(}->}[rd] & \P(X) \ar[r] \ar[d]^{e_{0,1}} & \P(X) / \tau \ar[r] \ar[d]_{e'_{0,1}} & X \ar@{^{(}->}[ld]\\
 & X\times X \ar[r] & \ss(X) 
}$$
where both slanted maps are diagonal inclusions, $e'_{0,1}$ is induced by $e_{0,1}$, $X$ is embedded in $\P(X)$ as the subspace of constant paths, and the right-most horizontal arrow sends the equivalence class of a path $\gamma$ to $\gamma(1/2)$. The diagram is strictly commutative, except for the right-hand side triangle, which commutes only up to homotopy (by contraction of paths toward their middle point). Note that $\TC^\Sigma(X)=\secat(e_{0,1})\geq\secat(e'_{0,1})$, the latter of which agrees with $\secat(X\hookrightarrow\ss(X))$, in view of the diagram.
\end{proof}

\begin{ejem}\label{tcstoro}{\em 
For a finite polyhedron $X$, the diagonal inclusion $X\hookrightarrow\ss(X)$ induces the trivial map in mod 2 positive-dimensional cohomology (see~\cite[Theorems~11.2 and~11.4]{MR0091462}). The usual nilker lower bound for $\secat(X\hookrightarrow\ss(X))$ then shows that $\TC^\Sigma(X)$ is bounded from below by the mod 2 cup-length of $\ss(X)$. In particular, Example~\ref{nakatoro} implies $3\leq\TC^\Sigma(S^1\times S^1)$. On the other hand, $\TC^\Sigma(S^1\times S^1)\leq\TC^S(S^1\times S^1)\leq 4$, in view of~\cite[Proposition~10]{MR2359031}. Note that deciding the sharp estimate for both $\TC^\Sigma(S^1\times S^1)$ and $\TC^S(S^1\times S^1)$ is decidable by (primary) obstruction-theoretic methods (in the equivariant setting, for $\TC^\Sigma$). It would be well worth taking a look at the actual needed computations, as this might lead to an example with $\TC^\Sigma\neq\TC^S$.
}\end{ejem}

\begin{corollary}\label{circle}
$\TC^\Sigma(S^1)=2$.
\end{corollary}
\begin{proof}
Recall $\TC^\Sigma(S^1)\leq\TC^S(S^1)=2$. If $\TC^\Sigma(S^1)\leq1$, the product inequality for $\TC^\Sigma$ (Lemma~\ref{proine} below) would yield $\TC^\Sigma(S^1\times S^1)\leq2$, which is impossible in view of Example~\ref{tcstoro}.
\end{proof}

The proof of~\cite[Theorem~11]{Far} can be used, word for word (using $\tau$-equivariant partitions of unit), to prove the auxiliary:
\begin{lema}\label{proine}
For paracompact spaces $X$ and $Y$, $\TC^\Sigma(X\times Y)\leq\TC^\Sigma(X)+\TC^\Sigma(Y)$.
\end{lema}


\bigskip\sc
Departamento de Matem\'aticas

Centro de Investigaci\'on y de Estudios Avanzados del IPN

Av.~IPN 2508, Zacatenco, M\'exico City 07000, M\'exico

{\tt jesus@math.cinvestav.mx}

\end{document}